\documentclass[a4paper,12pt]{article}

\usepackage{amscd,amssymb,amsmath,amsfonts,amsthm, mathrsfs}
\usepackage{graphicx}
\usepackage{verbatim,enumerate,paralist}
\usepackage{textcomp}
\usepackage[colorlinks]{hyperref}
\hypersetup{%
  urlcolor=blue,linkcolor=blue,citecolor=blue
}
\usepackage{pdfpages}

\usepackage[T1]{fontenc}
\usepackage[all,poly,knot,tips]{xy}
\newcounter{Definitioncount}
\newtheorem{theorem}{Theorem}% with counter

\newtheorem{nugget}[theorem]{Nugget}
\newtheorem{corollary}[theorem]{Corollary}
\theoremstyle{definition}
\newtheoremstyle{fact}{\bigskipamount}{\medskipamount}{\upshape}{}{\itshape}{. }{ }{Fact}
\theoremstyle{fact}
\newtheoremstyle{genquest}{\bigskipamount}{\medskipamount}{\upshape}{}{\itshape}{. }{ }{General Question}
\theoremstyle{genquest}
\newtheoremstyle{step}{2\bigskipamount}{\medskipamount}{\upshape}{}{\itshape}{. }{ }{\underline{Step~\thestep}}
\theoremstyle{step}

\renewcommand{\thestep}{\arabic{step}}

\newcommand{\lra}{\longrightarrow}

\newcommand{\Lra}{\Longrightarrow}

\newcommand{\ldual}[1]{\mathord{{\let\nolimits\relax\sideset{^\wedge}{}{#1}}}}
\newcommand{\laction}[2]{\mathord{{\let\nolimits\relax\sideset{^{#1}}{}{#2}}}}
\newcommand{\conj}[2]{\mathord{{\let\nolimits\relax\sideset{^{#1}}{}{#2}}}}

\def\CA{{\mathscr A}}
\def\CB{{\mathscr B}}
\def\CC{{\mathscr C}}

\def\CI{{\mathscr I}}

\def\CM{{\mathscr M}}
\def\CN{{\mathscr N}}

\def\CV{{\mathscr V}}

\def\CX{{\mathscr X}}

\begin{document}

\author{Ross Street\footnote{The author gratefully acknowledges the support of Australian Research Council Discovery Grant DP1094883.} \\ 
\small{Mathematics Department, Macquarie University, NSW 2109 Australia} \\
\small{<ross.street@mq.edu.au>}
}

\title{Kan extensions and cartesian monoidal categories}
\date{\small{18 September 2012}}
\maketitle

\noindent {\small{\emph{2010 Mathematics Subject Classification:} 18D10; 18D20; 18A40; 18C10}

\begin{abstract}
\noindent The existence of adjoints to algebraic functors between categories of models of 
Lawvere theories follows from finite-product-preservingness surviving left Kan extension.
A result along these lines was proved in Appendix 2 of Brian Day's PhD thesis \cite{DayPhD}.
His context was categories enriched in a cartesian closed base.
A generalization is described here with essentially the same proof.
We introduce the notion of cartesian monoidal category in the enriched context.
With an advanced viewpoint, we give a result about left extension along a 
promonoidal module and further related results.  
\end{abstract}

\tableofcontents

\section{Introduction}\label{Intro}

The pointwise left Kan extension, along any functor between categories with finite products,
of a finite-product-preserving functor into a cartesian closed category is finite-product-preserving.  
This kind of result goes back at least to Bill Lawvere's thesis \cite{LawverePhD} 
and some 1966 ETH notes of Fritz Ulmer. 
Eduardo Dubuc and the author independently provided Saunders Mac Lane with a proof along
the lines of the present note at Bowdoin College in the Northern Hemisphere Summer of 1969.
Brian Day's thesis \cite{DayPhD} gave a generalization to categories enriched in a cartesian closed base. Also see Kelly-Lack \cite{KelLack1993} and Day-Street \cite{DaySt1995}. 
Our purpose here is to remove the restriction on the base and, to some extent, the finite products.

\section{Weighted colimits}\label{Wc}

We work with a monoidal category $\CV$ as used in Max Kelly's book \cite{KellyBook} as a base for enriched category theory. 

Recall that the {\em colimit} of a $\CV$-functor $F:\CA \lra \CX$ {\em weighted by} a $\CV$-functor 
$W:\CA^{\mathrm{op}} \lra \CV$ is an object 
$$\mathrm{colim}(W,F) = \mathrm{colim}_A(WA,FA)$$ of $\CX$ 
equipped with an isomorphism
$$\CX(\mathrm{colim}(W,F),X)\cong [\CA^{\mathrm{op}},\CV ](W,\CX(F,X))$$
$\CV$-natural in $X$.

Independence of naturality in the two variables of two variable naturality, or Fubini's theorem \cite{KellyBook}, has the following expression in terms of weighted colimits.

\begin{nugget} \label{Fubini}  
For $\CV$-functors 
$$W_1:\CA_1^{\mathrm{op}} \lra \CV, \ W_2:\CA_2^{\mathrm{op}} \lra \CV , \ F:\CA_1 \otimes \CA_2 \lra \CX ,$$  
if $\mathrm{colim}(W_2,F(A,-))$ exists for each $A\in \CA$ then
$$\mathrm{colim}(W_1,\mathrm{colim}(W_2,F))\cong \mathrm{colim}(W_1\otimes W_2,F) \ .$$
Here the isomorphism is intended to include the fact that one side exists if and only if the other does. Also $(W_1\otimes W_2)(A,B)=W_1A\otimes W_2B$.  
\end{nugget}  
\begin{proof}  Here is the calculation:
\begin{eqnarray*}
\CX(\mathrm{colim}(W_1\otimes W_2,F),X) & \cong & [(\CA_1 \otimes \CA_2)^{\mathrm{op}},\CV ](W_1\otimes W_2 ,\CX (F,X))\\
& \cong & [\CA_1^{\mathrm{op}},\CV ](W_1, [\CA_2^{\mathrm{op}},\CV ](W_2 ,\CX (F,X)))\\
& \cong & [\CA_1^{\mathrm{op}},\CV ](W_1, \CX(\mathrm{colim}(W_2 ,F),X)))\\
& \cong & \CX(\mathrm{colim}(W_1,\mathrm{colim}(W_2,F)),X) \ .
\end{eqnarray*}
\end{proof}

Here is an aspect of the calculus of mates expressed in terms of weighted colimits. 
Note that $S\dashv T:\CA \lra \CC$ means $T^{\mathrm{op}}\dashv S^{\mathrm{op}}:\CA^{\mathrm{op}} \lra \CC^{\mathrm{op}}$. 

\begin{nugget}\label{mates}
For $\CV$-functors $W:\CA^{\mathrm{op}} \lra \CV$, $G:\CC \lra \CX$, and a $\CV$-adjunction
$$S\dashv T:\CA \lra \CC,$$
there is an isomorphism $$\mathrm{colim}(WS^{\mathrm{op}},G) \cong \mathrm{colim}(W,GT) \ . $$  
\end{nugget}
\begin{proof} Here is the calculation: 
\begin{eqnarray*}
\CX(\mathrm{colim}(W,GT),X) & \cong &  [\CA^{\mathrm{op}},\CV ](W,\CX (GT,X))\\
& \cong &  [\CA^{\mathrm{op}},\CV ](W,\CX (G,X)T^{\mathrm{op}})\\
& \cong &  [\CC^{\mathrm{op}},\CV ](WS^{\mathrm{op}},\CX (G,X))\\
& \cong &  \CX(\mathrm{colim}(WS^{\mathrm{op}},G),X) \ .
\end{eqnarray*}
\end{proof}

Recall that a {\em pointwise left Kan extension} of a $\CV$-functor $F:\CA \lra \CX$ along a
$\CV$-functor $J:\CA \lra \CB$ is a $\CV$-functor $K = \mathrm{Lan}_J(F) :\CB \lra \CX$
such that there is a $\CV$-natural isomorphism
$$KB\cong \mathrm{colim}_A(\CB(JA,B),FA) \ .$$ 

\section{Cartesian monoidal enriched categories}\label{Cec}

A monoidal $\CV$-category $\CA$ will be called {\em cartesian} when the tensor product and unit object have left adjoints. 
That is, $\CA$ is a map pseudomonoid in the monoidal 2-category $\CV\text{-}\mathrm{Cat}^{\mathrm{co}}$ in the sense of \cite{DMcCS}.   

Let us denote the tensor product of $\CA$ by $-\star - : \CA \otimes \CA \lra \CA$ with left adjoint 
$\Delta : \CA \lra \CA \otimes \CA$ and the unit by $N : \CI \lra \CA$ with left adjoint $E : \CA \lra \CI$. 
(Here $\CI$ is the unit $\CV$-category: it has one object $0$ and $\CI (0,0)=I$.)
It is clear that these right adjoints make $\CA$ a comonoidal $\CV$-category; that is, a pseudomonoid in
$\CV\text{-}\mathrm{Cat}^{\mathrm{op}}$. 
Since $\mathrm{ob} : \CV\text{-}\mathrm{Cat} \lra \mathrm{Set}$ is monoidal, we see that $\Delta : \CA \lra \CA \otimes \CA$ is given by the diagonal on objects.  We have
$$\CA (A,A_1\star A_2) \cong \CA (A,A_1)\otimes \CA (A,A_2) \ ,$$ 
where $\CV$-functoriality in $A$ on the right-hand side uses $\Delta$. 

If $\CA$ is cartesian, the $\CV$-functor category $[\CA,\CV]$ becomes monoidal under convolution using the comonoidal structure on $\CA$. 
This is a pointwise tensor product in the sense that, on objects, it is defined by: $$(M\ast N)A = MA\otimes NA \ .$$ 
On morphisms it requires the use of $\Delta$.
Indeed, the Yoneda embedding
$$\mathrm{Y} : \CA^{\mathrm{op}} \lra [\CA,\CV]$$
is strong monoidal.  

\section{Main result}\label{Mr}

\begin{theorem}
Suppose $J:\CA \lra \CB$ is a $\CV$-functor between cartesian monoidal $\CV$-categories.
Assume also that $J$ is strong comonoidal.
Suppose $\CX$ is a monoidal $\CV$-category such that each of the $\CV$-functors $-\otimes X$ and
$X\otimes -$ preserves colimits. Assume the $\CV$-functor $F:\CA \lra \CX$ is strong monoidal.
If the pointwise left Kan extension $K:\CB \lra \CX$ of $F$ along $J$ exists then $K$ too is strong monoidal.
\end{theorem}
\begin{proof} 
Using that tensor in $\CX$ preserves colimits in each variable, the Fubini Theorem~\ref{Fubini}, that $F$ is strong monoidal, Theorem~\ref{mates} with the cartesian property of $\CA$, and the cartesian property of $\CB$, we have the calculation:  
\begin{eqnarray*}
KB_1\otimes KB_2 & \cong & \mathrm{colim}_{A_1}(\CB(JA_1,B_1),FA_1)\otimes \mathrm{colim}_{A_2}(\CB(JA_2,B_2),FA_2) \\
& \cong & \mathrm{colim}_{A_1, A_2}(\CB(JA_1,B_1)\otimes \CB(JA_2,B_2),FA_1\otimes FA_2) \\  
& \cong & \mathrm{colim}_{A_1, A_2}(\CB(JA_1,B_1)\otimes \CB(JA_2,B_2),F(A_1\star A_2)) \\  
& \cong & \mathrm{colim}_{A}(\CB(JA,B_1)\otimes \CB(JA,B_2),FA) \\    
& \cong & \mathrm{colim}_{A}(\CB(JA,B_1\star B_2),FA) \\ 
& \cong & K(B_1 \star B_2) \ .  
\end{eqnarray*}

For the unit part, for similar reasons, we have:
\begin{eqnarray*}
N & \cong & FN0 \\
& \cong & \mathrm{colim}_0(\CI(0,0),FN0) \\  
& \cong & \mathrm{colim}_{A}(\CI(EA,0),FA) \\    
& \cong &  \mathrm{colim}_{A}(\CI(EJA,0),FA) \\ 
& \cong & KN \ .  
\end{eqnarray*}
\end{proof}

\section{An advanced viewpoint}\label{Aav}

In terminology of \cite{DaySt1997}, suppose $H:\CM \lra \CN$ is a monoidal pseudofunctor 
between monoidal bicategories. 
The main point to stress here is that the constraints 
$$\Phi_{A,B} : HA\otimes HB \lra H(A\otimes B)$$
are pseudonatural in $A$ and $B$.
Then we see that $H$ takes pseudomonoids (= monoidales) to pseudomonoids, 
lax morphisms of pseudomonoids to lax morphisms, 
oplax morphisms of pseudomonoids to oplax morphisms, and
strong morphisms of pseudomonoids to strong morphisms.

In particular, this applies to the monoidal pseudofunctor
$$\CV \text{-}\mathrm{Mod}(-,\CI) : \CV \text{-}\mathrm{Mod}^{\mathrm{op}} \lra \CV \text{-}\mathrm{CAT}$$ 
which takes the $\CV$-category $\CA$ to the $\CV$-functor $\CV$-category $[\CA,\CV]$.
Now pseudomonoids in $\CV \text{-}\mathrm{Mod}^{\mathrm{op}}$ are precisely promonoidal 
(= premonoidal) $\CV$-categories in the sense of Day \cite{DayPhD, DayConv}.  
Therefore, for each promonoidal $\CV$-category $\CA$, we obtain a monoidal $\CV$-category
$$\CV \text{-}\mathrm{Mod}(\CA,\CI) = [\CA,\CV]$$
which is none other than what is now called Day convolution since it is defined and analysed in
\cite{DayPhD, DayConv}.  

A lax morphism of pseudomonoids in $\CV \text{-}\mathrm{Mod}^{\mathrm{op}}$, as written in $\CV \text{-}\mathrm{Mod}$, is a module $K : \CB \lra \CA$ equipped with module morphisms 

$$\xymatrix{
\CB \ar[d]_{K}^(0.5){\phantom{AAA}}="1" \ar[rr]^{P}  && \CB \otimes \CB \ar[d]^{K\otimes K}_(0.5){\phantom{AAA}}="2" \ar@{<=}"1";"2"^-{\phi}
\\
\CA \ar[rr]_-{P} && \CA \otimes \CA 
}
\qquad
\xymatrix{
\CB \ar[rd]_{J}^(0.5){\phantom{a}}="1" \ar[rr]^{K}  && \CA \ar[ld]^{J}_(0.5){\phantom{a}}="2" \ar@{=>}"1";"2"^-{\phi_0}
\\
& \CI 
}
$$
satisfying appropriate conditions.
In other words, we have 
 \begin{eqnarray*}
& \phi_{A_1,A_2,B}  :  \mathrm{colim}_{B_1,B_2} (K(A_1,B_1)\otimes K(A_2,B_2),P(B_1,B_2,B)) \\
 & \Lra   \mathrm{colim}_A (K(A,B),P(A_1,A_2,A))
 \end{eqnarray*}
 and 
 \begin{eqnarray*}
 \phi_{0 B} :  JB \Lra \mathrm{colim}_A (K(A,B), JA) \ .
 \end{eqnarray*}
 We call such a $K$ a {\em promonoidal module}. It is {\em strong} when $\phi$ and $\phi_0$
 are invertible.
 
 We also have the $\CV$-functor 
 $$\exists_K : [\CA,\CX] \lra [\CB,\CX]$$
 defined by
 $$(\exists_K)B=\mathrm{colim}_A (K(A,B),FA) \ .$$
 
 By the general considerations on monoidal pseudofunctors, $\exists_K$ is a monoidal 
 $\CV$-functor when $\CX = \CV$.
 However, the same calculations needed to show this explicitly show that it works for any monoidal 
 $\CV$-category $\CX$ for which each of the tensors 
 $X\otimes -$ and $-\otimes X$ preserves colimits. 
 
 \begin{theorem}\label{HiT} If $K:\CB \lra \CA$ is a promonoidal $\CV$-module then $\exists_K : [\CA,\CX] \lra [\CB,\CX]$ is a monoidal $\CV$-functor.  If $K$ is strong promonoidal then $\exists_K$ is strong monoidal.  
 \end{theorem}
 \begin{proof} Although the result should be expected from our earlier remarks, here is a direct calculation.
 \begin{eqnarray*}
(\exists_KF_1\ast \exists_KF_2)B & \cong & \mathrm{colim}_{B_1,B_2}(P(B_1,B_2,B),(\exists_KF_1)B_1\otimes(\exists_KF_2)B_2) \\
& \cong & \mathrm{colim}_{B_1,B_2}(P(B_1,B_2,B), \mathrm{colim}_{A_1}(K(A_1,B_1), F_1A_1)\otimes \\
& &\mathrm{colim}_{A_2}(K(A_2,B_2), F_2A_2)) \\  
& \cong & \mathrm{colim}_{B_1,B_2,A_1,A_2}(K(A_1,B_1)\otimes K(A_2,B_2)\otimes P(B_1,B_2,B), \\
& & F_1A_1 \otimes F_2A_2) \\    
& \Lra &  \mathrm{colim}_{A,A_1,A_2}(K(A,B)\otimes P(A_1,A_2,A),F_1A_1\otimes F_2A_2) \\ 
& \cong & \mathrm{colim}_A(K(A,B), \mathrm{colim}_{A_1,A_2}(P(A_1,A_2,A),F_1A_1\otimes F_2A_2)) \\
& \cong & \mathrm{colim}_A(K(A,B), (F_1\ast F_2)A)) \\
& \cong & \exists_K(F_1 \ast F_2)B \ .
\end{eqnarray*}
The morphism on the fourth line of the calculation is induced by $\phi_{A_1,A_2,B}$ and so is invertible if $K$ is strong promonoidal. We also have $ \phi_{0 B} :  JB \Lra (\exists_KJ)B$. 
\end{proof}
 
 For the corollaries now coming, assume as above that $\CX$ is a monoidal $\CV$-category such that 
 $X\otimes -$ and $-\otimes X$ preserve existing colimits. 
 Also $\CA$ and $\CB$ are monoidal $\CV$-categories. 
 The monoidal structure on $[\CA^{\mathrm{op}},\CX]$ is convolution with respect to the
 promonoidal structure $\CA (A,A_1\star A_2)$ on $\CA^{\mathrm{op}}$; 
 similarly for $[\CB^{\mathrm{op}},\CX]$.
 
 \begin{corollary} If $J:\CA \lra \CB$ is strong monoidal then so is 
 $$\mathrm{Lan}_{J^{\mathrm{op}}} :[\CA^{\mathrm{op}},\CX]\lra [\CB^{\mathrm{op}},\CX] \ .$$
 \end{corollary} 
 \begin{proof}
 Apply Theorem~\ref{HiT} to the module $K : \CB^{\mathrm{op}} \lra \CA^{\mathrm{op}}$ defined by
 $K(A,B) = \CB (B,JA)$. We see that $K$ is strong promonoidal using Yoneda twice and strong monoidalness of $J$. 
 \end{proof}
 
 \begin{corollary} If $W:\CA \lra \CV$ is strong monoidal then so is 
$$\mathrm{colim}(W,-) :[\CA^{\mathrm{op}},\CX] \lra \CX \ .$$
 \end{corollary}
 \begin{proof}
 Take $\CB = \CI$ in Theorem~\ref{HiT}. 
 \end{proof}
 
\begin{corollary} \label{Cor3} Suppose $\CA$ is cartesian monoidal. 
If $F:\CA \lra \CX$ is strong monoidal then so is 
$$\mathrm{colim}(-,F) :[\CA^{\mathrm{op}},\CV] \lra \CX \ .$$
 \end{corollary}
 \begin{proof} Here is the calculation for binary tensoring:
 \begin{eqnarray*}
\mathrm{colim}(W_1\otimes W_2,F) & \cong & \mathrm{colim}_A((W_1\otimes W_2)\Delta A,FA) \\
& \cong & \mathrm{colim}_{A_1,A_2}(W_1A_1\otimes W_2A_2,F(A_1\star A_2)) \\ 
& \cong & \mathrm{colim}_{A_1,A_2}(W_1A_1\otimes W_2A_2,FA_1\otimes FA_2)) \\    
& \cong &  \mathrm{colim}_{A_1}(W_1A_1,FA_1) \otimes \mathrm{colim}_{A_2}(W_2A_2,FA_2) \\
& \cong &  \mathrm{colim}(W_1,F) \otimes \mathrm{colim}(W_2,F) \ .
\end{eqnarray*}
The unit preservation is easier.
 \end{proof}
 
 \begin{corollary} Suppose $\CA$ and $\CB$ are cartesian monoidal and $J:\CA \lra \CB$ is strong comonoidal. 
 If $F:\CA \lra \CX$ is strong monoidal then so is $$\mathrm{Lan}_JF :\CB\lra \CX \ .$$
 \end{corollary}
 \begin{proof}
 Notice that $\mathrm{Lan}_JF$ is the composite of $\CB(J,1):\CB \lra [\CA^{\mathrm{op}},\CV]$ and 
 $\mathrm{colim}(-,F) :[\CA^{\mathrm{op}},\CV] \lra \CX$. 
 The first is strong monoidal by hypothesis on $J$.
 The second is strong monoidal by Corollary~\ref{Cor3}.
  \end{proof}

\begin{center}
--------------------------------------------------------
\end{center}

\appendix

\end{document}